\documentclass[12pt]{article}
\usepackage{amsmath,amsfonts,amsthm,amssymb, mathtools}
\usepackage{mathrsfs, graphicx,color,latexsym, tikz, calc,cite,enumerate,indentfirst}
\usetikzlibrary{shadows}
\usetikzlibrary{patterns,arrows,decorations.pathreplacing}
\usepackage[colorlinks=true,
linkcolor=blue,citecolor=blue,
urlcolor=blue]{hyperref}

\newtheorem{theorem}{Theorem}[section]
\newtheorem{lemma}{Lemma}[section]
\newtheorem{problem}{Problem}[section]

\newtheorem{corollary}{Corollary}[section]
\newtheorem{remark}{Remark}[section]

\newtheorem{claim}{Claim}

\renewcommand\proofname{\it{Proof}}

\allowdisplaybreaks

\title{\bf Distance-regular Cayley graphs over dicyclic groups}

\author{{Xueyi Huang$^{a,b}$, Kinkar Chandra Das$^{b,}$\footnote{Corresponding author.}\setcounter{footnote}{-1}\footnote{\emph{E-mail address:} huangxymath@163.com (X. Huang), kinkardas2003@gmail.com (K. C. Das), lulumath@csu.edu.cn (L. Lu).}\ \ and Lu Lu$^{c}$}\\[2mm]
\small $^a$School of Mathematics, East China University of Science and Technology, \\
\small  Shanghai 200237, P. R. China\\
\small $^b$Department of Mathematics, Sungkyunkwan University,\\
\small Suwon 16419, Republic of Korea\\
\small $^c$School of Mathematics and Statistics, Central South University,\\
\small Changsha, Hunan, 410083, P. R. China
}

\date{}

\begin{document}
\maketitle

\begin{abstract}
The characterization of distance-regular Cayley graphs originated from the problem of identifying strongly regular Cayley graphs, or equivalently, regular partial difference sets. In this paper, a classification of  distance-regular Cayley graphs on dicyclic groups is obtained. More specifically, it is shown that  every distance-regular Cayley graph on a dicyclic group is a complete graph, a complete multipartite graph, or a non-antipodal bipartite distance-regular graph with diameter $3$ satisfying some additional conditions.

\par\vspace{2mm}

\par\vspace{1mm}

\end{abstract}

\section{Introduction}\label{section::1}

Let $G$ be a finite group with identity $1$, and let $S$ be a subset of $G\setminus\{1\}$ such that  $S=S^{-1}:=\{s^{-1}\mid s\in S\}$. The \textit{Cayley graph} $\mathrm{Cay}(G,S)$ is defined as the graph with vertex set $G$, and with an edge joining two vertices $g,h\in G$ if and only if $g^{-1}h\in S$. Here $S$ is called the \textit{connection set} of $\mathrm{Cay}(G,S)$. It is known that $\mathrm{Cay}(G,S)$ is connected if and only if  $\langle S \rangle=G$, and that $G$ acts regularly on the vertex set of $\mathrm{Cay}(G,S)$ by left multiplicity. If $G$ is a cyclic (resp. dihedral) group, then $\mathrm{Cay}(G,S)$ is called a \textit{circulant} (resp. \textit{dihedrant}).

Let $\Gamma$ be a connected graph with vertex set $V(\Gamma)$. The \textit{distance} $\partial_\Gamma(u,v)$ between two vertices $u,v$ of $\Gamma$ is the length of a shortest path connecting them in $\Gamma$, and the \textit{diameter} $d_\Gamma$ of $\Gamma$ is the maximum distance  in $\Gamma$. For $v\in V(\Gamma)$, let $N_i^\Gamma(v)$ denote the set of vertices at distance $i$ from $v$ in $\Gamma$. In particular, we denote $N^\Gamma(v)=N_1^\Gamma(v)$. When $\Gamma$ is clear from the context, we use  $\partial(u,v)$, $d$, $N_i(v)$ and $N(v)$ instead of  $\partial_\Gamma(u,v)$, $d_\Gamma$, $N_i^\Gamma(v)$ and $N^\Gamma(v)$, respectively. For  $u,v\in V(\Gamma)$ with $\partial(u,v)=i$ ($0\leq i\leq d$), let
$$
c_i(u,v)=|N_{i-1}(u)\cap N(v)|,~~ a_i(u,v)=|N_{i}(u)\cap N(v)|, ~~ b_i(u,v)=|N_{i+1}(u)\cap N(v)|.
$$
Here  $c_0(u,v)=b_d(u,v)=0$. Then $\Gamma$ is called \textit{distance-regular} if  $c_i(u,v)$, $b_i(u,v)$ and $a_i(u,v)$ do not depend on the choice of $u,v$ with $\partial(u,v)=i$, that is,  depend only on the distance $i$ between $u$ and $v$, for all $0\leq i\leq d$.  

For a distance-regular graph $\Gamma$ with diameter $d$, we denote  $c_i=c_i(u,v)$, $a_i=a_i(u,v)$ and $b_i=b_i(u,v)$, where $u,v\in V(\Gamma)$ with $\partial(u,v)=i$. By definition, $\Gamma$ is a regular graph with valency $k=b_0$, and
$a_i+b_i+c_i=k$ for $0\leq i\leq d$.  The array $\{b_0,b_1,\ldots,b_{d-1};c_1,c_2,\ldots,c_d\}$ is called the \textit{intersection array} of $\Gamma$. In particular, $\lambda=a_1$ is the number of common neighbors between two adjacent vertices in $\Gamma$, and $\mu=c_2$ is the number of common neighbors between two vertices  at distance $2$ in $\Gamma$.  A distance-regular graph on $n$ vertices with valency $k$ and diameter $2$ is also called a  \textit{strongly regular graph} with parameters $(n,k,\lambda=a_1,\mu=c_2)$. A distance-regular graph is called \textit{non-trivial} if it does not belong to any of the following classes: complete graphs, complete multipartite graphs, complete bipartite graphs without a $1$-factor, and cycles.

After observing some beautiful combinatorial properties of distance-transitive graphs, Biggs introduced the concept of  distance-regular graphs (see the monograph \cite{B74}  from 1974). In the past several decades, distance-regular graphs played an important role in the study of design theory and coding theory, and were closely linked to some other subjects such as finite group theory, representation theory, and association schemes. For more detailed results on combinatorial or algebraic properties of distance-regular graphs, we refer the reader to \cite{BCN89,DKT16}, and references therein.

As an extension of the problem of characterizing  strongly regular Cayley graphs (or equivalently, regular partial difference sets  \cite{M94}),  Miklavi\v{c} and Poto\v{c}nik  \cite{MP07} (see also \cite[Problem 71]{DKT16}) proposed the following problem.
\begin{problem}\label{prob::main}
For a class of groups $\mathcal{G}$, determine all distance-regular graphs, which are Cayley graphs on a group in $\mathcal{G}$.
\end{problem}

For strongly regular Cayley graphs, a classic work is that all strongly regular circulants were determined by Bridges and Mena \cite{BM79}, Ma \cite{M84}, and partially by Maru\v{s}i\v{c} \cite{M89}. Also, the strongly regular Cayley graphs on $\mathbb{Z}_{p^n}\times \mathbb{Z}_{p^n}$ were classified by Leifman and Muzychuck \cite{LM05}. However, the strongly regular Cayley graphs on general groups, even for abelian groups, are far from being completely characterized.

With regard to Problem  \ref{prob::main},  Miklavi\v{c} and Poto\v{c}nik \cite{MP03,MP07} (almost) classified the distance-regular circulants or dihedrants. Miklavi\v{c} and \v{S}parl \cite{MS14,MS20} characterized the distance-regular Cayley graphs  on abelian groups or generalized dihedral groups under the condition that the connection set is minimal with respect to some element. Abdollahi, van Dam and Jazaeri \cite{ADJ17} determined the distance-regular Cayley graphs of diameter at most three with least eigenvalue $-2$. Very recently, van Dam and Jazaeri \cite{DJ19,DJ21} determined some distance-regular Cayley graphs with small valency, and provided some characterizations for  bipartite distance-regular Cayley graphs with diameter  $3$ or $4$. 

Inspired by the work of Miklavi\v{c} and \v{S}parl \cite{MP03,MP07}, in this paper, we mainly focus on the characterization of distance-regular Cayley graphs on dicyclic groups.  For a positive integer $n$, the \textit{dicyclic group} $\mathrm{Dic}_n$ is defined by
\begin{equation*}
\mathrm{Dic}_n=\langle \alpha,\beta \mid \alpha^{2n}=1, \beta^2=\alpha^n, \beta^{-1}\alpha\beta=\alpha^{-1}\rangle.
\end{equation*}
Clearly, $\mathrm{Dic}_n=\langle\alpha\rangle\cup \langle\alpha\rangle\beta$ is a non-abelian group of order $4n$ if $n>1$, and $\mathrm{Dic}_n\cong\mathbb{Z}_4$ if $n=1$. Also, $\alpha^n$ is the unique element of order $2$ in $\mathrm{Dic}_n$. For simplicity, Cayley graphs on dicyclic groups are called  \textit{dicirculants}. For $j\in \mathbb{Z}_{2n}$ and  $A\subseteq \mathbb{Z}_{2n}$, we denote  $j+A=\{j+i\mid i\in A\}$, $jA=\{j\cdot i\mid i\in A\}$, $-A=\{-i\mid i\in A\}$, $\alpha^A=\{\alpha^i\mid i\in A\}$ and $\alpha^A\beta=\{\alpha^i\beta\mid i\in A\}$. Then every dicirculant on $4n$ vertices has the form $\mathrm{Cay}(\mathrm{Dic}_n,\alpha^R\cup \alpha^T\beta)$, where $R$ and $T$ are subsets of $\mathbb{Z}_{2n}$ such that  $0\not\in R$, $R=-R$ and $T=n+T$. For convenience, we denote $\mathrm{Dic}(n,R,T):=\mathrm{Cay}(\mathrm{Dic}_n,\alpha^R\cup \alpha^T\beta)$.  The main result is as follows.

\begin{theorem}\label{thm::main}
Let $\Gamma$ be a dicirculant on $4n$ vertices. Then $\Gamma$ is distance-regular if and only if it is isomorphic to one of the following graphs:
\begin{enumerate}[$(i)$]\setlength{\itemsep}{0pt}
\item the complete graph $K_{4n}$;
\item the complete multipartite graph $K_{t\times m}$  with $tm=4n$, which is the complement of the disjoint union of  $t$ 
copies of the complete graph $K_m$;
\item the graph $\mathrm{Dic}(n,R,T)$ for even $n$, where $R=-R$ and $T=n+T$ are non-empty subsets of  $1+2\mathbb{Z}_{2n}$ such that $|R\cap T|<n$ and $|R\cap (i+R)|+|T\cap (i+T)|=2|(j+R)\cap T|=2|R\cap T|$ for all $i,j\in 2\mathbb{Z}_{2n}\setminus\{0\}$.
\end{enumerate}
In particular, the graph in  $(iii)$ is a non-antipodal bipartite non-trivial distance-regular graph with  diameter $3$.
\end{theorem}

A subset $D$ of a group $G$ is called a \textit{difference set} if there is an integer $\mu$ such that for every $g\in G\setminus\{1\}$  the number of  $(g_1,g_2)\in D\times D$ satisfying  $g_2g_1^{-1}=g$ is equal to $\mu$. If $|D|\notin\{|G|, |G|-1, 1, 0\}$, then $D$ is \textit{non-trivial}. In review of \cite[Lemma 2.8]{MP07}, it is not difficult to verify that the statement in Theorem \ref{thm::main} (iii) is actually equivalent to:
\begin{enumerate}
\item[$(iii')$] the graph $\mathrm{Dic}(n,R,T)$ for even $n$, where $R=-R$ and $T=n+T$ are non-empty subsets of  $1+2\mathbb{Z}_{2n}$ such that  $\alpha^{-1+R}\cup \alpha^{-1+T}\beta$ is a non-trivial difference set in the dicyclic group $\langle\alpha^2,\beta\rangle$ of order $2n$.
\end{enumerate}

By Theorem \ref{thm::main}, we obtain the following corollary immediately.
\begin{corollary}\label{cor::main}
Let $\Gamma$ be a  dicirculant on $4n$ vertices where $n$ is odd. Then $\Gamma$ is distance-regular if and only if it is isomorphic to the complete graph $K_{4n}$, or the complete multipartite graph $K_{t\times m}$, where $tm=4n$.
\end{corollary}

\section{Preliminaries}\label{section::2}

In this section, we review some notations and results about distance-regular graphs, which are powerful in the proof of Theorem \ref{thm::main}.

Let $\Gamma$ be a graph, and let $\mathcal{B}=\{B_1,\ldots,B_r\}$ be a partition of $V(\Gamma)$. The \textit{quotient graph} of $\Gamma$ with respect to $\mathcal{B}$, denoted by $\Gamma_\mathcal{B}$, is the graph with vertex set $\mathcal{B}$, and with  $B_i,B_j$ ($i\neq j$) adjacent if and only if there exists at least one edge between  $B_i$ and $B_j$ in $\Gamma$. Moreover, we say that $\mathcal{B}$ is an \textit{equitable partition} of $\Gamma$ if there are integers $b_{ij}$ ($1\leq i,j\leq r$) such that every vertex in $B_i$ has exactly $b_{ij}$ neighbors in $B_j$.

Suppose that $\Gamma$ is a distance-regular graph with diameter $d$. For  $i\in\{1,\ldots,d\}$, the \textit{$i$-th distance graph} $\Gamma_i$  is  the graph with vertex set  $V(\Gamma)$ in which two distinct vertices are adjacent if and only if they are at distance $i$ in $\Gamma$. We say that $\Gamma$ is \textit{primitive} if $\Gamma_i$ is connected for all $i\in\{1,\ldots,d\}$, and \textit{imprimitive} otherwise.  Also, we say that  $\Gamma$ is \textit{antipodal} if the relation $\mathcal{R}$ on $V(\Gamma)$ defined by $u\mathcal{R}v\Leftrightarrow\partial(u,v)\in\{0,d\}$ is an equivalence relation. It is  known that an imprimitive distance-regular graph with valency at least $3$ is either bipartite, antipodal, or both \cite[Theorem 4.2.1]{BCN89}.

If $\Gamma$ is a bipartite distance-regular graph, then $\Gamma_2$ has two connected components, which are called the \textit{halved graphs} of $\Gamma$ and denoted by $\Gamma^+$ and $\Gamma^-$. For convenience, we use  $\frac{1}{2}\Gamma$ to represent any one of these two graphs. If $\Gamma$ is an antipodal distance-regular graph, then the  relation $\mathcal{R}$ defined  above leads to a partition $\mathcal{B}^\ast$ of $V(\Gamma)$ into equivalence classes, called \textit{fibres}.  It is known that  $\mathcal{B}^\ast$ is actually  an equitable partition of $\Gamma$, and all fibres of $\Gamma$ share the same size.  The  \textit{antipodal quotient} of $\Gamma$, denoted by $\overline{\Gamma}$, is defined as the quotient graph  $\Gamma_{\mathcal{B}^\ast}$. Let $r$ be the common size of fibres of $\Gamma$. Then  $\Gamma$ is said to be an $r$-\textit{fold antipodal cover} of $\overline{\Gamma}$. Note that if $d=2$ then $\Gamma$ is a complete multipartite graph, and that if $d\geq 3$ then the edges between  two distinct fibres of $\Gamma$ form an empty set or a $1$-factor.

\begin{lemma}(\cite[Proposition 4.2.2]{BCN89}) \label{lem::imprimitive}
Let $\Gamma$ denote an imprimitive distance-regular graph with diameter $d$ and valency $k \geq 3$. Then the following hold.
\begin{enumerate}[$(i)$]\setlength{\itemsep}{0pt}
\item If $\Gamma$ is bipartite, then the halved graphs of $\Gamma$ are non-bipartite distance-regular graphs with diameter $\lfloor\frac{d}{2}\rfloor$.
\item If $\Gamma$  is antipodal, then $\overline{\Gamma}$ is a distance-regular graph with diameter $\lfloor\frac{d}{2}\rfloor$.
\item  If $\Gamma$ is antipodal, then $\overline{\Gamma}$ is not antipodal, except when $d\leq 3$  (in that case $\overline{\Gamma}$  is a complete graph), or when $\Gamma$ is bipartite with $d = 4$ (in that case $\overline{\Gamma}$ is a complete bipartite graph).

\item If $\Gamma$ is antipodal and has odd diameter or is not bipartite, then $\overline{\Gamma}$ is primitive.
\item If $\Gamma$ is bipartite and has odd diameter or is not antipodal, then the halved graphs of $\Gamma$ are primitive.
\item If $\Gamma$ has even diameter and is both bipartite and antipodal, then $\overline{\Gamma}$ is bipartite. Moreover, if $\frac{1}{2}\Gamma$ is a halved graph of $\Gamma$, then it is antipodal, and $\overline{\frac{1}{2}\Gamma}$ is primitive and isomorphic to $\frac{1}{2}\overline{\Gamma}$.
\end{enumerate}
\end{lemma}

\begin{lemma}(\cite[Theorem 6.2]{GH92})\label{lem::quotient_graph}
Let $\Gamma$ be an antipodal distance-regular graph with diameter $d\geq 3$, and let $\mathcal{B}$ be an equitable partition of $\Gamma$ with each block contained in a
fibre of $\Gamma$. Assume that no block of $\mathcal{B}$ is a single vertex, or a fibre. Then all
blocks  of  $\mathcal{B}$ have the same size, and the  quotient graph $\Gamma_{\mathcal{B}}$ is an
antipodal distance-regular graph with diameter $d$. Moreover, $\Gamma$ and $\Gamma_{\mathcal{B}}$ have isomorphic antipodal quotients.
\end{lemma}

\begin{lemma}(\cite[p.425, p.431]{BCN89})\label{lem::antipodal_DRG}
Let $\Gamma$ be an $r$-fold antipodal distance-regular graph on $n$ vertices with  diameter $d$ and valency $k$.
\begin{enumerate}[$(i)$]\setlength{\itemsep}{0pt}
\item If $\Gamma$ is non-bipartite and $d=3$, then $n=r(k+1)$, $k=\mu(r-1)+\lambda+1$, and $\Gamma$ has  the intersection array $\{k,\mu(r-1),1;1,\mu,k\}$ and the spectrum $\{k^1,\theta_1^{m_1},\theta_2^k,\theta_3^{m_3}\}$, where
\begin{equation*}
\theta_1=\frac{\lambda-\mu}{2}+\delta,~~\theta_2=-1,~~\theta_3=\frac{\lambda-\mu}{2}-\delta,~~\delta=\sqrt{k+\left(\frac{\lambda-\mu}{2}\right)^2},
\end{equation*}
and
\begin{equation*}
m_1=-\frac{\theta_3}{\theta_1-\theta_3}(r-1)(k+1),~~m_3=\frac{\theta_1}{\theta_1-\theta_3}(r-1)(k+1).
\end{equation*}

\item If $\Gamma$ is bipartite and $d=4$, then $n=2r^2\mu$, $k=r\mu$, and $\Gamma$ has the intersection array $\{r\mu, r\mu-1,(r-1)\mu, 1;1,\mu,r\mu-1,r\mu\}$.
\end{enumerate}
\end{lemma}

A \textit{conference graph} is a strongly regular graph with parameters $(n,k=\frac{n-1}{2},\lambda=\frac{n-5}{4},\mu=\frac{n-1}{4})$, where $n\equiv 1\pmod 4$. Paley graphs, introduced  by Sachs \cite{S62},  and independently by Erdős and Rényi \cite{ER63}, form an infinite family of conference graphs. Let $\mathbb{F}_q$ denote the finite field of order $q$  where $q\equiv 1\pmod 4$ is a prime power. The \textit{Paley graph} $P(q)$ is defined as the graph with vertex set $\mathbb{F}_q$ in which two distinct vertices $u,v\in \mathbb{F}_q$ are adjacent if and only if $u-v$ is a square in the multiplicative group of $\mathbb{F}_q$.

\begin{lemma}(\cite[p. 180]{BCN89})\label{lem::confer}
Let $\Gamma$ be a conference graph (or particularly, Paley graph). Then:
\begin{enumerate}[$(i)$]\setlength{\itemsep}{0pt}
\item $\Gamma$ has no distance-regular $r$-fold antipodal covers for $r>1$, except for the pentagon $C_5\cong P(5)$, which is covered by the decagon $C_{10}$;
\item $\Gamma$ cannot be a halved graph of a bipartite distance-regular graph.
\end{enumerate}
\end{lemma}

Recall that circulants are Cayley graphs on cyclic groups. In \cite{MP03}, Miklavi\v{c} and Poto\v{c}nik determined all (primitive) distance-regular circulants.

\begin{lemma}(\cite[Theorem 1.2, Corollary 3.7]{MP03}) \label{lem::cir_DRG}
Let $\Gamma$ be a circulant on $n$ vertices. Then $\Gamma$ is distance-regular if and only if it is isomorphic to one of the following graphs:
\begin{enumerate}[$(i)$]\setlength{\itemsep}{0pt}
\item the cycle $C_n$;
\item the complete graph $K_n$;
\item the complete multipartite graph $K_{t\times m}$, where $tm=n$;
\item the complete bipartite graph without a $1$-factor $K_{m,m}-mK_2$, where $2m = n$, $m$ odd;
\item the Paley graph $P(n)$, where $n\equiv 1\pmod 4$  is prime.
\end{enumerate}
In particular, $\Gamma$ is a primitive distance-regular graph if and only if $\Gamma\cong K_n$, or $n$ is prime, and $\Gamma\cong C_n$ or $P(n)$.
\end{lemma}

Also, Miklavi\v{c} and Poto\v{c}nik  gave a characterization of primitive distance-regular Cayley graphs in terms of  distance module and Schur ring (see  \cite{MP03} for the definition).

\begin{lemma}(\cite[Proposition 3.6]{MP03}) \label{lem::Schur_DRG}
Let $\Gamma=\mathrm{Cay}(G,S)$  denote a distance-regular Cayley graph and $\mathcal{D}=\mathcal{D}_\mathbb{Z}(G,S)$ its distance module. Then:
\begin{enumerate}[$(i)$]\setlength{\itemsep}{0pt}
\item $\mathcal{D}$ is a primitive Schur ring over $G$ if and only if $\Gamma$ is a primitive distance-regular graph;
\item $\mathcal{D}$ is the trivial Schur ring  over $G$ if and only if $\Gamma$ is isomorphic to the complete graph.
\end{enumerate}
\end{lemma}

Recall that $\mathrm{Dic}_n\cong \mathbb{Z}_4$ if $n=1$. According to \cite[Theorem 4]{S57} and \cite[Theorem 3.4]{MP09}, we have the following result.

\begin{lemma}\label{lem::Schur_dic}
For every $n\geq 1$, there are no non-trivial primitive Schur rings over the dicyclic group $\mathrm{Dic}_n$.
\end{lemma}

Recall that Cayley graphs on dicyclic groups are called dicirculants.   If $\Gamma$ is  a primitive distance-regular dicirculant on $4n$ vertices, then its distance module would be a primitive Schur ring over $\mathrm{Dic}_n$ by Lemma \ref{lem::Schur_DRG} (i), and hence can only be  the trivial Schur ring  by Lemma \ref{lem::Schur_dic}. Therefore, Lemma \ref{lem::Schur_DRG} (ii) implies the following result.

\begin{corollary}\label{cor::pri_DRG}
Let $\Gamma$ be a primitive distance-regular dicirculant on $4n$ vertices. Then $\Gamma$ is isomorphic to the complete graph $K_{4n}$.
\end{corollary}

Let $G$ be a transitive permutation group acting on a set $X$. An \textit{imprimitivity system}  for $G$ is a partition $\mathcal{B}$ of $X$ which is invariant under the action of $G$, i.e., for every block $B\in \mathcal{B}$ and for every $g\in G$, we have $B^g=B$ or $B^g\cap B=\emptyset$.

\begin{lemma}(\cite[Lemma 2.2]{MP07}) \label{lem::block}
Let $\Gamma=\mathrm{Cay}(G,S)$ denote a Cayley graph with the group $G$ acting regularly on the vertex set of $\Gamma$ by left multiplication. Suppose there exists an imprimitivity system $\mathcal{B}$ for $G$. Then the block $B\in\mathcal{B}$ containing the identity  $1\in G$ is a subgroup in $G$. Moreover,
\begin{enumerate}[$(i)$]\setlength{\itemsep}{0pt}
\item if $B$ is normal in $G$, then $\Gamma_\mathcal{B}=\mathrm{Cay}(G/B,S/B)$, where $S/B=\{sB\mid s\in S\setminus B\}$;
\item if there exists an abelian subgroup $A$ in $G$ such that $G = AB$, then $\Gamma_\mathcal{B}$ is isomorphic to a
Cayley graph on the group $A/(A\cap B)$.
\end{enumerate}
\end{lemma}

By Lemmas \ref{lem::imprimitive} and \ref{lem::block}, we obtain the following corollary.

\begin{corollary}\label{cor::DRG_dic}
Let $\Gamma$ denote a distance-regular dicirculant.
\begin{enumerate}[$(i)$]\setlength{\itemsep}{0pt}
\item If $\Gamma$ is antipodal, then the antipodal quotient $\overline{\Gamma}$ is a distance-regular circulant or a distance-regular dicirculant.
\item If $\Gamma$ is bipartite, then the halved graphs $\Gamma^+$ and $\Gamma^-$ are distance-regular circulants or distance-regular dicirculants.
\end{enumerate}
\end{corollary}
\begin{proof}
Let  $\Gamma$ be defined on   $\mathrm{Dic}_n=\langle \alpha,\beta \mid \alpha^{2n}=1, \beta^2=\alpha^n, \beta^{-1}\alpha\beta=\alpha^{-1}\rangle$.  First assume that $\Gamma$ is antipodal. Since $\mathrm{Dic}_n$ acts regularly on the vertex set of $\Gamma$ by left multiplication,  the antipodal classes of $\Gamma$ form an imprimitivity system $\mathcal{B}$ for $\mathrm{Dic}_n$. Let $B\in \mathcal{B}$ denote the antipodal class of $\Gamma$ containing the identity of $\mathrm{Dic}_n$.  By Lemma \ref{lem::block}, $B$ is a subgroup of $\mathrm{Dic}_n$. If $B$ is a subgroup of $\langle\alpha\rangle$, then $B$ is  normal in $\mathrm{Dic}_n$, and it follows from  Lemma \ref{lem::block} (i) that $\overline{\Gamma}=\Gamma_\mathcal{B}=\mathrm{Cay}(G/B,S/B)$, which is  a dicirculant. If $B$ is not a subgroup of $\langle\alpha\rangle$, then $\mathrm{Dic}_n=\langle\alpha\rangle B$, and Lemma \ref{lem::block} (ii) implies that $\overline{\Gamma}=\Gamma_\mathcal{B}$ is isomorphic to a Cayley graph on the group $\langle\alpha\rangle/(\langle\alpha\rangle\cap B)$. Hence, $\overline{\Gamma}$ is a circulant. Now assume that $\Gamma$ is bipartite. Let $\Gamma^+$ denote the halved graph containing the identity of $\mathrm{Dic}_n$. Since the bipartition sets of $\Gamma$ form an imprimitivity system for $\mathrm{Dic}_n$, again by Lemma \ref{lem::block}, $V(\Gamma^+)$ is a subgroup of $\mathrm{Dic}_n$. It is easy to see that $V(\Gamma^+)$  acts regularly on itself by left multiplication as a subgroup of $\mathrm{Aut}(\Gamma^+)$. Since every subgroup of $\mathrm{Dic}_n$ is cyclic or dicyclic (cf. \cite{RD16}), $V(\Gamma^+)$ is a circulant or a dicirculant. Moreover, since $\Gamma$ is vertex transitive, the two halved graphs $\Gamma^+$ and $\Gamma^-$ are isomorphic. Hence, $\Gamma^-$ is  also a circulant or a dicirculant. Note that $\overline{\Gamma}$, $\Gamma^+$ and $\Gamma^-$ are distance-regular by Lemma \ref{lem::imprimitive}. The result follows.
\end{proof}

Let $n$ be a positive integer, and let $\omega$ be a primitive $n$-th root of unity. Let  $\mathbb{F}=\mathbb{Q}(\omega)$ denote the $n$-th cyclotomic field over the rationals. For a subset $A\subseteq \mathbb{Z}_n$, let $\Delta_A:\mathbb{Z}_n\rightarrow \mathbb{F}$ be the \textit{characteristic function} of $A$, that is, $\Delta_A(z)=1$ if $z\in A$, and $\Delta_A(z)=0$ otherwise. In particular, if $A=\{a\}$, then we write $\Delta_a$ instead of $\Delta_{\{a\}}$.
Let $\mathbb{F}^{\mathbb{Z}_n}$ be the $\mathbb{F}$-vector space consisting of all functions $f:\mathbb{Z}_n\rightarrow \mathbb{F}$ with the scalar multiplication and addition defined point-wise. We denote by $(\mathbb{F}^{\mathbb{Z}_n},\cdot)$ the  $\mathbb{F}$-algebra obtained from $\mathbb{F}^{\mathbb{Z}_n}$ by defining the multiplication point-wise, and $(\mathbb{F}^{\mathbb{Z}_n},\ast)$ the $\mathbb{F}$-algebra obtained from $\mathbb{F}^{\mathbb{Z}_n}$ by defining the multiplication as the \textit{convolution}  (see \cite{MP07}):
\begin{equation}\label{equ::fourier}
(f\ast g)(z)=\sum_{i\in \mathbb{Z}_n}f(i)g(z-i),~~f,g\in \mathbb{F}^{\mathbb{Z}_n}.
\end{equation}
The \textit{Fourier transformation} $\mathcal{F}:(\mathbb{F}^{\mathbb{Z}_n},\ast)\rightarrow(\mathbb{F}^{\mathbb{Z}_n},\cdot)$ is defined by
\begin{equation}\label{equ::fourier1}
(\mathcal{F}f)(z)=\sum_{i\in \mathbb{Z}_n}f(i)\omega^{iz},~~f\in \mathbb{F}^{\mathbb{Z}_n}.
\end{equation}
It is easy to verify that $\mathcal{F}$ is an algebra isomorphism from  $(\mathbb{F}^{\mathbb{Z}_n},\ast)$ to $(\mathbb{F}^{\mathbb{Z}_n},\cdot)$.

Let $\mathbb{Z}_n^\ast=\{i\in \mathbb{Z}_n\mid \mathrm{gcd}(i,n)=1\}$ denote the multiplicative group of units in the ring  $\mathbb{Z}_n$. Then  $\mathbb{Z}_n^\ast$ acts on $\mathbb{Z}_n$ by multiplication. It is known that each orbit of this action consists of all elements of a given order in the additive group $\mathbb{Z}_n$. Consequently, each orbit is of  the form  $O_r=\{c\cdot \frac{n}{r}\in \mathbb{Z}_n\mid c\in \mathbb{Z}_n^\ast\}$, where $r$ is some positive divisor of $n$.

The following three lemmas present some basic facts about Fourier transformation.

\begin{lemma}(\cite[Corollary 3.2]{MP07})\label{lem::Fourier2}
If $A$ is a subset of $\mathbb{Z}_{n}$ and $\mathrm{Im}(\mathcal{F}\Delta_A)\subseteq\mathbb{Q}$, then $A$ is a union of some orbits of the action of $\mathbb{Z}_n^\ast$ on $\mathbb{Z}_n$ by multiplication, and $\mathrm{Im}(\mathcal{F}\Delta_A)\subseteq\mathbb{Z}$.
\end{lemma}

\begin{lemma}(\cite[Lemma 3.3]{MP07})\label{lem::Fourier4}
 Let $r$ be a  positive divisor  of $n$, and let $\omega$ be a primitive $n$-th root of unity. If $A$ is a subset of $\mathbb{Z}_{n}$, then
$$\mathcal{F}\Delta_{A}\left(\frac{n}{r}\right)=e_0+e_1\xi+\cdots+e_{r-1}\xi^{r-1},$$ 
where $\xi=\omega^{\frac{n}{r}}$ and $e_i =|A\cap (i+r\mathbb{Z}_n)|$ for $0\leq i\leq r-1$.
\end{lemma}

Let $H$ be a subgroup of $G$. A \textit{transversal} of  $H$  is a subset of $G$ that contains exactly one element from each of the right cosets of $H$ in $G$.

\begin{lemma}(\cite[Lemma 3.4]{MP07})\label{lem::Fourier1}
Let $r$ be a positive divisor of $n$, and let $A$ be a transversal of the subgroup $r\mathbb{Z}_n$ in $\mathbb{Z}_n$. If $z=m\frac{n}{r}$ ($m\not\in r\mathbb{Z}_n$) is an arbitrary element of $\frac{n}{r}\mathbb{Z}_n\setminus\{0\}$, then $\mathcal{F}\Delta_A(z)=0$.
\end{lemma}

\begin{lemma}(\cite[Lemma 4.3]{MP07})\label{lem::Fourier3}
Let $p$ be a  prime divisor of $n$, and let $A$ be a transversal of the subgroup $\frac{n}{p}\mathbb{Z}_n$ in $\mathbb{Z}_n$. If $A$ is a union of some orbits of the action of $\mathbb{Z}_n^\ast$ on $\mathbb{Z}_n$ by multiplication, then $p = 2$ or $A = p\mathbb{Z}_n$.
\end{lemma}

\section{The classification of distance-regular dicirculants}\label{section::3}

The main goal of this section is to prove  Theorem \ref{thm::main}, which gives a classification of distance-regular dicirculants. For  simplicity, we keep the following notation.
{\flushleft\bf Notation.} 
Denote by $\mathrm{Dic}_n=\langle \alpha,\beta \mid \alpha^{2n}=1, \beta^2=\alpha^n, \beta^{-1}\alpha\beta=\alpha^{-1}\rangle$  the dicyclic group of order $4n$. Suppose that $\Gamma=\mathrm{Dic}(n,R,T)=\mathrm{Cay}(\mathrm{Dic}_n,\alpha^R\cup \alpha^T\beta)$ is a distance-regular dicirculant, where $R,T$ are subsets of $\mathbb{Z}_{2n}$ such that $0\not\in R$, $R=-R$ and $T=n+T$. Note that  $T\neq \emptyset$ because $\Gamma$ is connected.  Denote by $k$, $\lambda$, $\mu$ and $d$ the valency, the number of common neighbors of two adjacent vertices, the number of common neighbors of two vertices at distance $2$, and the diameter of $\Gamma$, respectively. For  $j\in\{0,1,\ldots,d\}$, let $\mathcal{N}_j=N_j(1)$ denote the set of vertices  at distance $j$ from the identity vertex $1\in\mathrm{Dic}_n$ in $\Gamma$, and let $R_j=\{i\in \mathbb{Z}_{2n}\mid \alpha^i\in \mathcal{N}_j\}$ and $T_j=\{i\in \mathbb{Z}_{2n}\mid \alpha^i\beta\in \mathcal{N}_j\}$. Clearly, $R_0=\{0\}$, $T_0=\emptyset$, $R_1=R$ and $T_1=T$.

Before giving the proof of Theorem \ref{thm::main}, we first set down a sequence of lemmas.

\begin{lemma}\label{lem::neighbor}
Let  $\Gamma=\mathrm{Dic}(n,R,T)$ be a  dicirculant. Then $N(\alpha^i)=\alpha^{i+R}\cup \alpha^{i+T}\beta$ and $N(\alpha^i\beta)=\alpha^{i-T}\cup \alpha^{i+R}\beta$.
\end{lemma}
\begin{proof}
By definition, we have $N(\alpha^i)=\alpha^i(\alpha^R\cup \alpha^T\beta)=\alpha^{i+R}\cup \alpha^{i+T}\beta$ and $N(\alpha^i\beta)=\alpha^i\beta(\alpha^R\cup \alpha^T\beta)=\alpha^{i-R}\beta \cup \alpha^{i-T}\beta^2=\alpha^{i+R}\beta \cup \alpha^{i+n-T}=\alpha^{i+R}\beta \cup \alpha^{i+2n-(n+T)}=\alpha^{i-T}\cup \alpha^{i+R}\beta$ because $R=-R$ and $T=n+T$.
\end{proof}

\begin{lemma}\label{lem::common_neighbor}
Let  $\Gamma=\mathrm{Dic}(n,R,T)$ be a  dicirculant. Then $|N(\alpha^i)\cap N(\alpha^j)|=|N(\alpha^i\beta)\cap N(\alpha^j\beta)|=|R\cap (j-i+R)|+|T\cap (j-i+T)|$, and $|N(\alpha^i)\cap N(\alpha^j\beta)|=2|(j-i+R)\cap T|$.
\end{lemma}
\begin{proof}
Recall that $R=-R$ and $T=n+T$. By Lemma \ref{lem::neighbor}, we have
$|N(\alpha^{i})\cap N(\alpha^{j})|=|(\alpha^{i+R}\cup \alpha^{i+T}\beta) \cap (\alpha^{j+R}\cup \alpha^{j+T}\beta)|=|(i+R)\cap (j+R)|+|(i+T)\cap (j+T)|=|R\cap (j-i+R)|+|T\cap (j-i+T)|$ and $|N(\alpha^{i}\beta)\cap N(\alpha^{j}\beta)|=|(\alpha^{i+R}\beta \cup \alpha^{i-T}) \cap (\alpha^{j+R}\beta \cup \alpha^{j-T})|=|(i+R)\cap (j+R)|+|(i-T)\cap (j-T)|=|R\cap (j-i+R)|+|T\cap (j-i+T)|$. Similarly,  $|N(\alpha^{i})\cap N(\alpha^{j}\beta)|=|(\alpha^{i+R}\cup \alpha^{i+T}\beta) \cap (\alpha^{j+R}\beta \cup \alpha^{j-T})|=|(i+R)\cap (j-T)|+|(j+R)\cap (i+T)|=2|(j-i+R)\cap T|$, as desired.
\end{proof}

\begin{lemma}\label{lem::intersec}
Let  $\Gamma=\mathrm{Dic}(n,R,T)$ be a distance-regular dicirculant. Then  
$$
|N(\alpha^i\beta)\cap \alpha^R|=|N(\alpha^i\beta)\cap \alpha^T\beta|=
\left\{
\begin{array}{ll}
\frac{\lambda}{2}, & \mbox{if $i\in T$;}\\[1mm]
\frac{\mu}{2}, & \mbox{if $i\in T_2$.}\\
\end{array}\right.
$$
In particular, $\lambda$ is even, and $\mu$ is even whenever $T_2\neq \emptyset$.
\end{lemma}
\begin{proof}
By Lemma \ref{lem::neighbor},  $N(\alpha^i\beta)\cap N(1)=N(\alpha^i\beta)\cap(\alpha^R\cup \alpha^T\beta)=(N(\alpha^i\beta)\cap \alpha^R)\cup (N(\alpha^i\beta)\cap \alpha^T\beta)=(\alpha^{i-T}\cap \alpha^R)\cup (\alpha^{i-R}\beta\cap \alpha^{T}\beta)$. Since $|(i-T)\cap R|=|(i-R)\cap T|$, we deduce that
$$
|N(\alpha^i\beta)\cap N(1)|=2|N(\alpha^i\beta)\cap \alpha^R|=2|N(\alpha^i\beta)\cap \alpha^T\beta|.
$$
Note that $|N(\alpha^i\beta)\cap N(1)|=\lambda$ if $i\in T$ and $|N(\alpha^i\beta)\cap N(1)|=\mu$ if $i\in T_2$. The result follows.
\end{proof}
\begin{lemma}\label{lem::special_pair}
Let  $\Gamma=\mathrm{Dic}(n,R,T)$ be a distance-regular dicirculant. Then $|N(1)\cap N(\alpha^n)|\geq |T|$. In particular, $\lambda\geq |T|$ if $n\in R$, and $\mu\geq |T|$ if $n\not\in R$.
\end{lemma}
\begin{proof}
By Lemma \ref{lem::common_neighbor}, $|N(1)\cap N(\alpha^n)|=|R\cap (n+R)|+|T\cap (n+T)|=|R\cap (n+R)|+|T|\geq |T|$. Note that $1$ and $\alpha^n$ are adjacent if $n\in R$, and at distance $2$ if $n\not\in R$. The result follows.
\end{proof}

Let $\omega=e^{\pi \mathbf{i}/n}$ be the primitive $2n$-th root of unity, and let  $\mathbb{F}=\mathbb{Q}(\omega)$. Suppose that  $(\mathbb{F}^{\mathbb{Z}_{2n}},\cdot)$ and $(\mathbb{F}^{\mathbb{Z}_{2n}},\ast)$ are $\mathbb{F}$-algebras defined as in Section \ref{section::2}, and that $\mathcal{F}$ is the Fourier transformation  from $(\mathbb{F}^{\mathbb{Z}_{2n}},\ast)$ to $(\mathbb{F}^{\mathbb{Z}_{2n}},\cdot)$ defined as in  \eqref{equ::fourier1}. We denote
\begin{equation}\label{equ::fourier5}
\underline{\mathbf{r}}_j(z)=(\mathcal{F}\Delta_{R_j})(z)=\sum_{i\in R_j}\omega^{iz}~~\mbox{and}~~ \underline{\mathbf{t}}_j(z)=(\mathcal{F}\Delta_{T_j})(z)=\sum_{i\in T_j}\omega^{iz},
\end{equation}
where $\Delta_{R_j}$ and $\Delta_{T_j}$ are the characteristic functions of $R_j$ and $T_j$, respectively. In particular, we denote $\underline{\mathbf{r}}=\underline{\mathbf{r}}_1=\mathcal{F}\Delta_{R}$ and $\underline{\mathbf{t}}=\underline{\mathbf{t}}_1=\mathcal{F}\Delta_{T}$. Let $\ast$ be the convolution of   $(\mathbb{F}^{\mathbb{Z}_{2n}},\ast)$ defined  as in \eqref{equ::fourier}. For  $A,B\subseteq \mathbb{Z}_{2n}$, we can verify that
\begin{equation}\label{equ::fourier2}
(\Delta_A\ast \Delta_B)(i)=|(i-A)\cap B|=|(i-B)\cap A|,~~i\in\mathbb{Z}_{2n}.
\end{equation}

\begin{lemma}\label{lem::Fourier}
Let $\Gamma=\mathrm{Dic}(n,R,T)$ be a distance-regular dicirculant. Then  $\underline{\mathbf{r}}^2+|\underline{\mathbf{t}}|^2=k+\lambda \underline{\mathbf{r}}+\mu\underline{\mathbf{r}}_2$ and $2\underline{\mathbf{rt}}=\lambda\underline{\mathbf{t}}+\mu\underline{\mathbf{t}}_2$.
\end{lemma}

\begin{proof}
By Lemma \ref{lem::neighbor} and \eqref{equ::fourier2}, for every $i\in \mathbb{Z}_{2n}$,
\begin{equation}\label{equ::fourier3}
\begin{aligned}
(\Delta_R\ast \Delta_R)(i)+(\Delta_T\ast \Delta_{-T})(i)&=|R\cap (i-R)|+|T\cap (i+T)|\\
&=(k\Delta_0+\lambda\Delta_R+\mu\Delta_{R_2})(i)\\
\end{aligned}
\end{equation}
and
\begin{equation}\label{equ::fourier4}
\begin{aligned}
2(\Delta_R\ast \Delta_T)(i)&=|R\cap (i-T)|+|T\cap (i-R)|\\
&=|R\cap (i+n-T)|+|T\cap (i-R)|\\
&=(\lambda \Delta_T+\mu\Delta_{T_2})(i).
\end{aligned}
\end{equation}
Recall that the Fourier transformation $\mathcal{F}$ is an algebra isomorphism  from $(\mathbb{F}^{\mathbb{Z}_{2n}},\ast)$ to $(\mathbb{F}^{\mathbb{Z}_{2n}},\cdot)$. By applying $\mathcal{F}$ on both sides of \eqref{equ::fourier3} and \eqref{equ::fourier4}, we obtain $\underline{\mathbf{r}}^2+|\underline{\mathbf{t}}|^2=k+\lambda \underline{\mathbf{r}}+\mu\underline{\mathbf{r}}_2$ and $2\underline{\mathbf{rt}}=\lambda\underline{\mathbf{t}}+\mu\underline{\mathbf{t}}_2$, respectively.
\end{proof}

\begin{lemma}\label{lem::key0}
 There is no distance-regular dicirculant  isomorphic to a complete bipartite graph without a $1$-factor.
\end{lemma}
\begin{proof}
Let $\Gamma=\mathrm{Dic}(n,R,T)$ be a distance-regular dicirculant.  By contradiction, assume that $\Gamma\cong K_{2n,2n}-2nK_2$. Then the valency of $\Gamma$ is equal to $2n-1$. However, this is only possible when $n\in R$ because $\alpha^n$ is the unique element of order $2$ in $\mathrm{Dic}_n$.  In this situation,  it follows from Lemma \ref{lem::special_pair} that  $\lambda\geq |T|$. As  $T\neq \emptyset$, the graph $\Gamma$ contains a triangle, which is impossible.
\end{proof}

\begin{lemma}\label{lem::key1}
There are no antipodal non-bipartite distance-regular dicirculants with diameter $3$.
\end{lemma}
\begin{proof}
By contradiction, assume that $\Gamma=\mathrm{Dic}(n,R,T)$ is an antipodal non-bipartite distance-regular dicirculant of diameter $3$ with the minimum order. Let $k$ and $p$ denote the valency and the common size of antipodal classes (or fibres) of $\Gamma$, respectively. Note that $p\geq 2$ and $k\geq 4$. In fact, if $k=2$ or $3$, since $\Gamma$ is connected and $\alpha^n$ is the unique element of order $2$ in $\mathrm{Dic}_n$, we have $n=1$ and  $\Gamma\cong C_4$ or $K_4$, which is impossible. According to Lemma \ref{lem::antipodal_DRG} (i),  $k+1=\frac{4n}{p}$,  and $\Gamma$ has the intersection array
\begin{equation}\label{equ::1.2}
\{k,\mu(p-1),1;1,\mu,k\}
\end{equation}
and eigenvalues $k$, $\theta_1$, $\theta_2=-1$, $\theta_3$, where
\begin{equation}\label{equ::1.3}
\theta_1=\frac{\lambda-\mu}{2}+\delta,~~\theta_3=\frac{\lambda-\mu}{2}-\delta~~\mbox{and}~~\delta=\sqrt{k+\left(\frac{\lambda-\mu}{2}\right)^2}.
\end{equation}

Let $H=\mathcal{N}_3\cup\{1\}$. Then $H$ is an antipodal class of $\Gamma$, and  $|H|=p$. Since $\mathrm{Dic}_n$ acts regularly on  $V(\Gamma)$ by left multiplication, the antipodal classes of $\Gamma$ form an imprimtivity system for  $\mathrm{Dic}_n$. By Lemma \ref{lem::block}, $H$ is a subgroup of $\mathrm{Dic}_n$. If $p$ is not  prime, then $H$ has  a non-trivial subgroup $K$  contained in $\langle\alpha\rangle$. Let $\mathcal{B}$ denote the partition consisting of all orbits of $K$ acting on $V(\Gamma)$ by left multiplication.  As $K$ is normal in $\mathrm{Dic}_n$,  the partition $\mathcal{B}$ is also an imprimtivity system for $\mathrm{Dic}_n$, and it follows from  Lemma \ref{lem::block} (i) that the quotient graph $\Gamma_\mathcal{B}$ is a dicirculant. Observe that $\mathcal{B}$ is  an equitable partition of $\Gamma$, and each block of $\mathcal{B}$ is contained in some fibre of $\Gamma$ and is neither a single vertex nor a fibre.  By Lemma \ref{lem::quotient_graph}, $\Gamma_\mathcal{B}$ is an antipodal distance-regular graph with diameter $3$. If $\Gamma_\mathcal{B}$ is bipartite, then $\Gamma$ is also bipartite, a contradiction. Hence, $\Gamma_\mathcal{B}$ is an antipodal non-bipartite distance-regular dicirculant of diameter $3$ with smaller order than $\Gamma$, contrary to our assumption. Therefore, $p$ is a prime number.  Moreover, we assert that $\mathcal{N}_3\subseteq\langle \alpha\rangle$. In fact, if $\mathcal{N}_3\cap \langle \alpha\rangle\beta\neq \emptyset$, the group $H$ would contain some element of order $4$, and hence $4\mid p$, which is impossible. Since $H=\mathcal{N}_3\cup \{1\}$ is the subgroup of $\langle \alpha\rangle$ with order $p$, we have $p\mid 2n$ and $\mathcal{N}_3=\{\alpha^{i\frac{2n}{p}}\mid i=1,2,\ldots,p-1\}$. Hence,  $R_3= \frac{2n}{p}\mathbb{Z}_{2n}\setminus\{0\}$ and $T_3=\emptyset$.  Before going further, similarly as  in \cite[Lemma 4.4]{MP07}, we need the following claim.

\begin{claim}\label{claim::1}
The sets  $R\cup \{0\}$ and $T$ are transversals of the subgroup $\frac{2n}{p}\mathbb{Z}_{2n}$ in $\mathbb{Z}_{2n}$. In particular, $p\neq 2$ and $p\mid n$.
\end{claim}
\renewcommand\proofname{\it{Proof of Claim \ref{claim::1}}}
\begin{proof}
First assume that $|T\cap (\ell+\frac{2n}{p}\mathbb{Z}_{2n})|\geq 2$ for some $\ell\in \mathbb{Z}_{2n}$. Then there exists some $i\in\{1,\ldots,p-1\}$ such that $i\frac{2n}{p}\in T-T$. Thus  $\alpha^{i\frac{2n}{p}}\in\mathcal{N}_1\cup \mathcal{N}_2$, contrary to  $\alpha^{i\frac{2n}{p}}\in \mathcal{N}_3$. Hence, $|T\cap (\ell+\frac{2n}{p}\mathbb{Z}_{2n})|\leq 1$ for all $\ell\in \mathbb{Z}_{2n}$. Similarly, $|(R\cup \{0\})\cap (\ell+\frac{2n}{p}\mathbb{Z}_{2n})|\leq 1$ for all $\ell\in \mathbb{Z}_{2n}$. Now assume that $T\cap (\ell+\frac{2n}{p}\mathbb{Z}_{2n})=\emptyset$ for some  $\ell\in \mathbb{Z}_{2n}$. Then $\ell+\frac{2n}{p}\mathbb{Z}_{2n}\subseteq T_2$ due to $T_0=T_3=\emptyset$. Since each vertex of $\mathcal{N}_2$ has a neighbor in $\mathcal{N}_3$, there exists some $i\in\{1,\ldots,p-1\}$ such that $\alpha^{i\frac{2n}{p}}\in \mathcal{N}_3$ is adjacent to $\alpha^{\ell+\frac{2n}{p}}\beta\in \mathcal{N}_2$. This implies that  $\ell+(1-i)\frac{2n}{p}\in T$, which contradicts $T\cap (\ell+\frac{2n}{p}\mathbb{Z}_{2n})=\emptyset$. Hence, $T$ has non-empty intersection with every coset of $\frac{2n}{p}\mathbb{Z}_{2n}$ in $\mathbb{Z}_{2n}$. Similarly, $R\cup \{0\}$ has non-empty intersection with every coset of $\frac{2n}{p}\mathbb{Z}_{2n}$ in $\mathbb{Z}_{2n}$. Therefore, we conclude that  $T$ and $R\cup \{0\}$ are transversals of the subgroup $\frac{2n}{p}\mathbb{Z}_{2n}$ in $\mathbb{Z}_{2n}$. This proves the first part of the claim. For the second part, suppose to the contrary that $p=2$.  Recall that $T=n+T$ is non-empty. For any $i\in T$, there exists some $\ell\in\mathbb{Z}_{2n}$ such that $i\in \ell+\frac{2n}{p}\mathbb{Z}_{2n}=\ell+n\mathbb{Z}_{2n}$. Then  $n+i\in n+\ell+n\mathbb{Z}_{2n}=\ell+n\mathbb{Z}_{2n}$. As $n+i\in n+T=T$, we get $|T\cap (\ell+n\mathbb{Z}_{2n})|\geq 2$, which is impossible by above arguments. Therefore, $p\neq 2$, and hence $p\mid n$.
\end{proof}

By Claim \ref{claim::1}, $|R|=\frac{2n}{p}-1$ and $|T|=\frac{2n}{p}$. Since $R_2=\mathbb{Z}_{2n}\setminus(\frac{2n}{p}\mathbb{Z}_{2n}\cup R)$ and $T_2=\mathbb{Z}_{2n}\setminus T$, we have $|R_2|=(p-1)|R|$ and $|T_2|=(p-1)|T|$. Furthermore, from \eqref{equ::fourier5} we obtain $\underline{\mathbf{r}}_2=2n\Delta_0-p\Delta_{p\mathbb{Z}_{2n}}-\underline{\mathbf{r}}$
and $\underline{\mathbf{t}}_2=2n\Delta_0-\underline{\mathbf{t}}$. Thus, by Lemma \ref{lem::Fourier},
\begin{equation}\label{equ::1.4}
\left\{
\begin{aligned}
&\underline{\mathbf{r}}^2+|\underline{\mathbf{t}}|^2=k+(\lambda-\mu)\underline{\mathbf{r}}-p\mu\Delta_{p\mathbb{Z}_{2n}}+2n\mu\Delta_0,\\[1mm]
&2\underline{\mathbf{rt}}=(\lambda-\mu)\underline{\mathbf{t}}+2n\mu\Delta_0.
\end{aligned}
\right.
\end{equation}
Clearly, $\underline{\mathbf{r}}(0)=|R|=\frac{2n}{p}-1$. Moreover, by Lemma \ref{lem::Fourier1} and Claim \ref{claim::1}, we have $\underline{\mathbf{r}}(z)=-1$ for all $z\in p\mathbb{Z}_{2n}\setminus\{0\}$. Now suppose $z\not\in p\mathbb{Z}_{2n}$. By  \eqref{equ::1.4},   if $\underline{\mathbf{t}}(z)\neq 0$ then $\underline{\mathbf{r}}(z)=\frac{\lambda-\mu}{2}$, and  if $\underline{\mathbf{t}}(z)=0$ then $\underline{\mathbf{r}}(z)\in\{\theta_1,\theta_3\}$, where $\theta_1=\frac{\lambda-\mu}{2}+\delta$ and $\theta_3=\frac{\lambda-\mu}{2}-\delta$ are the two eigenvalues of $\Gamma$ given in \eqref{equ::1.3}. Putting $B=\{z\in\mathbb{Z}_{2n}\mid z\not\in p\mathbb{Z}_{2n},~\underline{\mathbf{t}}(z)=0,~\underline{\mathbf{r}}(z)=\theta_1\}$, $C=\{z\in\mathbb{Z}_{2n}\mid z\not\in p\mathbb{Z}_{2n},~\underline{\mathbf{t}}(z)=0,~\underline{\mathbf{r}}(z)=\theta_3\}$, and $D=\mathbb{Z}_{2n}\setminus(B\cup C\cup p\mathbb{Z}_{2n})$. Then
\begin{equation}\label{equ::1.5}
\underline{\mathbf{r}}(z)=\left\{
\begin{array}{ll}
\frac{2n}{p}-1, & z=0,\\[1mm]
-1, & z\in p\mathbb{Z}_{2n}\setminus\{0\},\\[1mm]
\theta_1, & z\in B,\\[1mm]
\theta_3, & z\in C,\\[1mm]
\frac{\lambda-\mu}{2}, & z\in D.
\end{array}
\right.
\end{equation}
If $\delta\in \mathbb{Q}$, then \eqref{equ::1.5} implies that $\mathrm{Im}(\underline{\mathbf{r}})\subseteq \mathbb{Q}$. By Lemma \ref{lem::Fourier2}, $R$ is a union of some orbits of  $\mathbb{Z}_{2n}^\ast$ acting on $\mathbb{Z}_{2n}$, and so is $R\cup \{0\}$. According to Claim \ref{claim::1} and Lemma \ref{lem::Fourier3}, we conclude that $R\cup \{0\}=p\mathbb{Z}_{2n}$, i.e., $R=p\mathbb{Z}_{2n}\setminus\{0\}$. Then, for each $i\in R$, $N(\alpha^i)\cap \alpha^{R_2}=\emptyset$, or equivalently, $N(\alpha^i)\cap \mathcal{N}_2\subseteq \alpha^{T_2}\beta$.  On the other hand, by Lemma \ref{lem::intersec},  $|N(\alpha^i\beta)\cap \alpha^R|=\frac{\mu}{2}$ for each $i\in T_2$. By counting the edges between $\alpha^R$ and $\alpha^{T_2}\beta$ in two ways and by \eqref{equ::1.2}, we have $|R|\mu(p-1)=|T_2|\frac{\mu}{2}$. Combining this with $|R|=\frac{2n}{p}-1$ and $|T_2|=(p-1)|T|=(p-1)\frac{2n}{p}$, we obtain $n=p$, and hence $|R|=1$ and $|T|=2$. Thus $k=|R|+|T|=3$, contrary to $k\geq 4$. If $\delta\not\in \mathbb{Q}$, the two eigenvalues $\theta_1,\theta_3$ of $\Gamma$ must have the same multiplicity, and hence $\lambda=\mu=\frac{k-1}{p}$ by Lemma \ref{lem::antipodal_DRG} (i). On the other hand, from Lemma \ref{lem::special_pair} and  $k+1=\frac{4n}{p}$, we deduce that $\lambda=\mu\geq |T|=\frac{2n}{p}=\frac{k+1}{2}$. Thus, we have $p<2$, a contradiction.

Therefore, we conclude that there are no antipodal non-bipartite distance-regular dicirculants with diameter $3$.
\end{proof}

\begin{lemma}\label{lem::key2}
There are no antipodal bipartite distance-regular dicirculants with diameter $4$.
\end{lemma}
\begin{proof}
By contradiction, assume that  $\Gamma=\mathrm{Dic}(n,R,T)$ is an antipodal bipartite distance-regular dicirculant of diameter $4$ with the minimum order. Let $k$ and $p$ denote the valency and the common size of antipodal classes  of $\Gamma$, respectively. Note that $p\geq 2$ and $k\geq 4$.  Also, by Lemma \ref{lem::antipodal_DRG} (ii),
\begin{equation}\label{equ::2.1}
2n=p^2\mu~~\mbox{and}~~k=p\mu.
\end{equation}
Similarly as in the proof of Lemma \ref{lem::key1}, we assert that $p$ is prime and  $\mathcal{N}_4=\{\alpha^{i\frac{2n}{p}}\mid i=1,2,\ldots,p-1\}=\alpha^{\frac{2n}{p}\mathbb{Z}_{2n}}\setminus\{1\}$. 
Since the bipartition set $\mathcal{N}_0\cup \mathcal{N}_2\cup \mathcal{N}_4$ is a subgroup of $\mathrm{Dic}_n$ with index $2$, we have $\mathcal{N}_0\cup \mathcal{N}_2\cup \mathcal{N}_4=\langle\alpha\rangle$, or $n$ is even and $\mathcal{N}_0\cup \mathcal{N}_2\cup \mathcal{N}_4\in\{\alpha^{2\mathbb{Z}_{2n}}\cup \alpha^{2\mathbb{Z}_{2n}}\beta,\alpha^{2\mathbb{Z}_{2n}}\cup\alpha^{1+2\mathbb{Z}_{2n}}\beta\}$. If  $\mathcal{N}_0\cup \mathcal{N}_2\cup \mathcal{N}_4=\langle\alpha\rangle$, then  $R=\emptyset$, and  by Lemma \ref{lem::special_pair}, $\mu\geq |T|=|R|+|T|=k$. Combining this with \eqref{equ::2.1} yields that $p=1$, a contradiction. Therefore,  $n$ is even and $\mathcal{N}_0\cup \mathcal{N}_2\cup \mathcal{N}_4\in\{\alpha^{2\mathbb{Z}_{2n}}\cup \alpha^{2\mathbb{Z}_{2n}}\beta,\alpha^{2\mathbb{Z}_{2n}}\cup\alpha^{1+2\mathbb{Z}_{2n}}\beta\}$.

Observe that $\mathrm{Dic}(n,R,T)\cong \mathrm{Dic}(n,R,1+T)$. We may assume that $\mathcal{N}_0\cup \mathcal{N}_2\cup \mathcal{N}_4=\alpha^{2\mathbb{Z}_{2n}}\cup \alpha^{2\mathbb{Z}_{2n}}\beta$. Since $\mathcal{N}_0\cup \mathcal{N}_4=\alpha^{\frac{2n}{p}\mathbb{Z}_{2n}}$ and $T_0=\emptyset$, we have $T_4=\emptyset$ and $T_2=2\mathbb{Z}_{2n}$, and hence $\mu$ is even by Lemma \ref{lem::intersec}. Furthermore, $R_2=2\mathbb{Z}_{2n}\setminus\frac{2n}{p}\mathbb{Z}_{2n}$ and $T\cup T_3=R\cup R_3=1+2\mathbb{Z}_{2n}$. Since every pair of vertices in $\mathcal{N}_4$ are at distance $4$, the set $\alpha^{R_3}$ partitions into subsets $\alpha^{i+R}$, $i\in R_4$, and hence $|R_3|=(p-1)|R|$. Then from \eqref{equ::2.1} and $|R_3|=n-|R|$ we deduce that $|R|=\frac{p\mu}{2}$. Similarly, the set  $\alpha^{T_3}\beta$ partitions into subsets $\alpha^{i+T}\beta$, $i\in R_4$, and hence $|T_3|=(p-1)|T|$, which gives that $|T|=\frac{p\mu}{2}$. As $n\not\in R$, by Lemma \ref{lem::special_pair}, $\mu\geq |T|$, and hence $p=2$. Therefore,  $|R|=|T|=\mu=\frac{n}{2}$ and $k=|R|+|T|=n$. Since  $R_2=2\mathbb{Z}_{2n}\setminus n\mathbb{Z}_{2n}$ and $T_2=2\mathbb{Z}_{2n}$, we have $\underline{\mathbf{r}}_2=n\Delta_{n\mathbb{Z}_{2n}}-2\Delta_{2\mathbb{Z}_{2n}}$ and $\underline{\mathbf{t}}_2=n\Delta_{n\mathbb{Z}_{2n}}$. By Lemma \ref{lem::Fourier}, 
\begin{equation}\label{equ::2.2}
\left\{
\begin{aligned}
&\underline{\mathbf{r}}^2+|\underline{\mathbf{t}}|^2=n+\frac{n^2}{2}\Delta_{n\mathbb{Z}_{2n}}-n\Delta_{2\mathbb{Z}_{2n}},\\
&2\underline{\mathbf{rt}}=\frac{n^2}{2}\Delta_{n\mathbb{Z}_{2n}}.
\end{aligned}
\right.
\end{equation}
Clearly, $\underline{\mathbf{r}}(0)=\underline{\mathbf{t}}(0)=\frac{n}{2}$ and $\underline{\mathbf{r}}(n)=-|\underline{\mathbf{t}}(n)|=-\frac{n}{2}$. By \eqref{equ::2.2},  $\underline{\mathbf{r}}(z)=\underline{\mathbf{t}}(z)=0$ for all $z\in 2\mathbb{Z}_{2n}\setminus\{0,n\}$. Moreover, if $z\not\in 2\mathbb{Z}_{2n}$, then $\underline{\mathbf{r}}(z)=0$ or $\underline{\mathbf{t}}(z)=0$. For the former case, $|\underline{\mathbf{t}}(z)|=\sqrt{n}$, and for the later case, $\underline{\mathbf{r}}(z)\in\{\sqrt{n},-\sqrt{n}\}$. Putting $B=\{z\in\mathbb{Z}_{2n}\mid z\not\in 2\mathbb{Z}_{2n},~ \underline{\mathbf{t}}(z)=0,~\underline{\mathbf{r}}(z)=\sqrt{n}\}$, $C=\{z\in\mathbb{Z}_{2n}\mid z\not\in 2\mathbb{Z}_{2n}, ~\underline{\mathbf{t}}(z)=0,~\underline{\mathbf{r}}(z)=-\sqrt{n}\}$, and $D=\mathbb{Z}_{2n}\setminus(2\mathbb{Z}_{2n}\cup B\cup C)$. Then we have
\begin{equation}\label{equ::2.3}
\underline{\mathbf{r}}(z)=\left\{
\begin{array}{ll}
\frac{n}{2},&z=0,\\[1mm]
-\frac{n}{2},&z=n,\\[1mm]
0,& z\in 2\mathbb{Z}_{2n}\setminus\{0,n\},\\[1mm]
\sqrt{n},&z\in B,\\[1mm]
-\sqrt{n},&z\in C,\\[1mm]
0,&z\in D,\\[1mm]
\end{array}
\right.
~~\mbox{and}~~
|\underline{\mathbf{t}}(z)|=\left\{
\begin{array}{ll}
\frac{n}{2},&z=0,\\[1mm]
\frac{n}{2},&z=n,\\[1mm]
0,& z\in 2\mathbb{Z}_{2n}\setminus\{0,n\},\\[1mm]
0,&z\in B,\\[1mm]
0,&z\in C,\\[1mm]
\sqrt{n},&z\in D.\\
\end{array}
\right.
\end{equation}
For $E\in\{R,T\}$, let $\underline{\mathbf{e}}=\mathcal{F}\Delta_E$. Let $t$ be a positive integer such that $2^t\mid 2n$. For each $i\in \{0,1,\ldots,2^t-1\}$, let $E_i(t)=E\cap (i+2^t\mathbb{Z}_{2n})$, and $e_i(t)=|E_i(t)|$. Since $E\subseteq 1+2\mathbb{Z}_{2n}$, we see that
\begin{equation}\label{equ::2.4}
e_0(t)=e_2(t)=\cdots=e_{2^t-2}(t)=0.
\end{equation}
Similarly as in  \cite[Lemma 4.5]{MP07}, we need the following claim.
\begin{claim}\label{claim::2}
For any integer $t\geq 2$, if $\underline{\mathbf{e}}(\frac{2n}{2^i})=0$ for all $i\in\{2,\ldots,t\}$, then $e_1(t)=e_3(t)=\cdots=e_{2^t-1}(t)$.
\end{claim}
\renewcommand\proofname{\it{Proof of Claim \ref{claim::2}}}
\begin{proof}
Let $\omega=e^{\pi \mathbf{i}/n}$ denote the primitive $2n$-th root of unity. First assume that $t=2$. If $\underline{\mathbf{e}}(\frac{2n}{4})=\underline{\mathbf{e}}(\frac{n}{2})=0$, by Lemma \ref{lem::Fourier4} and \eqref{equ::2.4}, we obtain $\underline{\mathbf{e}}(\frac{n}{2})=(e_1(2)-e_3(2))\omega^{\frac{n}{2}}=(e_1(2)-e_3(2))\mathbf{i}$. Hence, $e_1(2)=e_3(2)$, and the results follows. Now take $t\geq 3$, and assume that the result holds for any $t'$ with $2\leq t'\leq t-1$. Suppose that $\underline{\mathbf{e}}(\frac{2n}{2^i})=0$ for all $i\in\{2,\ldots,t\}$. Again by Lemma \ref{lem::Fourier4}, 
\begin{equation}\label{equ::2.5}
0=\underline{\mathbf{e}}\left(\frac{2n}{2^t}\right)=\sum_{i=0}^{2^{t-1}-1}(e_i(t)-e_{i+2^{t-1}}(t))\xi^i,
\end{equation}
where $\xi=\omega^{\frac{2n}{2^t}}$. Note that the degree of the minimal polynomial for $\xi$ over $\mathbb{Q}$ is $\varphi(2^t)=2^{t-1}$, where $\varphi(\cdot)$ denotes the Euler totient function. Then \eqref{equ::2.5} implies that
\begin{equation}\label{equ::2.6}
e_i(t)=e_{i+2^{t-1}}(t),~\mbox{for all}~i\in\{0,1,\ldots,2^{t-1}-1\}.
\end{equation}
On the other hand, by induction  hypothesis, there exists some integer $c$ such that
\begin{equation}\label{equ::2.7}
e_1(t-1)=e_3(t-1)=\cdots=e_{2^{t-1}-1}(t-1)=c.
\end{equation}
Since  $2^{t-1}\mathbb{Z}_{2n}$ is the disjoint union of $2^t\mathbb{Z}_{2n}$ and $2^{t-1}+2^t\mathbb{Z}_{2n}$,  $E_i(t-1)$ is the disjoint union of $E_i(t)$ and $E_{i+2^{t-1}}(t)$, and hence $e_i(t-1)=e_i(t)+e_{i+2^{t-1}}(t)$.  Combining this with \eqref{equ::2.6} and \eqref{equ::2.7}, we obtain
$$
e_1(t)=e_3(t)=\cdots=e_{2^t-1}(t)=\frac{c}{2},
$$
and the result follows.
\end{proof}

Recall that $4|E|=4|R|=4|T|=4\mu=2n$. By \eqref{equ::2.4} and Claim \ref{claim::2},  if $\underline{\mathbf{e}}(\frac{2n}{2^i})=0$ for all $i\in\{2,\ldots,t\}$, then  $2^{t-1}$ would be a divisor of $|E|$, and hence  $2^{t+1}\mid 2n$. Let $s$ be the largest integer such that $2^s\mid 2n$. Since $2n=4\mu$ and $\mu$ is even, $s\geq 3$. By \eqref{equ::2.3}, $\underline{\mathbf{r}}(\frac{2n}{2^i})=\underline{\mathbf{t}}(\frac{2n}{2^i})=0$ for all $i\in\{2,3,\ldots,s-1\}$. Moreover, since $\frac{2n}{2^s}\not\in\{0,n\}$, $\underline{\mathbf{r}}(\frac{2n}{2^s})=0$ or $\underline{\mathbf{t}}(\frac{2n}{2^s})=0$. Thus we can choose suitable $E\in\{R,T\}$ such that  $\underline{\mathbf{e}}(\frac{2n}{2^i})=0$ for all $i\in\{2,\ldots,s\}$. However, this implies that $2^{s+1}\mid 2n$, contrary to the maximality of $s$.

Therefore, we conclude that there are no antipodal bipartite distance-regular dicirculants with diameter $4$.
\end{proof}

Now we are in a position to give the proof of Theorem \ref{thm::main}.

\renewcommand\proofname{\it{Proof of Theorem \ref{thm::main}}}
\begin{proof} 
First of all, we  show that the graphs listed in (i)--(iii) are distance-regular dicirculants. For (i): the complete graph $K_{4n}$ is distance-regular with diameter $1$, and $K_{4n}\cong \mathrm{Cay}(\mathrm{Dic}_n,\mathrm{Dic}_n\setminus\{1\})$. For (ii):  the complete multipartite graph $K_{t\times m}$ ($tm=4n$)  is distance-regular with diameter $2$. Since $m$  is a divisor of  $4n$, by elementary group theory, there exists a subgroup $H^{(m)}$ of order $m$ in $\mathrm{Dic}_n$ (cf. \cite{RD16}). Then it is easy to see that  $K_{t\times m}\cong \mathrm{Cay}(\mathrm{Dic}_n, \mathrm{Dic}_n\setminus H^{(m)})$. Now consider (iii). Suppose  that $R=-R$ and $T=n+T$ are non-empty subsets of  $1+2\mathbb{Z}_{2n}$ ($n$ is even)  such that $|R\cap T|<n$ and 
\begin{equation}\label{equ::intersec_set}
|R\cap (i+R)|+|T\cap (i+T)|=2|(j+R)\cap T|=2|R\cap T|
\end{equation}
for all $i,j\in 2\mathbb{Z}_{2n}\setminus\{0\}$. 
Let $\Gamma=\mathrm{Dic}(n,R,T)=\mathrm{Cay}(\mathrm{Dic}_n,\alpha^R\cup \alpha^T\beta)$ and  $H=\langle\alpha^2,\beta\rangle=\alpha^{2\mathbb{Z}_{2n}}\cup \alpha^{2\mathbb{Z}_{2n}}\beta$. Clearly, $\Gamma$ is a bipartite graph where the bipartition is given by $H\cup \alpha H$. Let $i_1,i_2\in 2\mathbb{Z}_{2n}$.   Since $i_2-i_1\in 2\mathbb{Z}_{2n}$, by Lemma \ref{lem::common_neighbor} and \eqref{equ::intersec_set},  we have
$|N(\alpha^{i_1})\cap N(\alpha^{i_2})|=|N(\alpha^{i_1}\beta)\cap N(\alpha^{i_2}\beta)|=|R\cap (i_2-i_1+R)|+|T\cap (i_2-i_1+T)|=2|R\cap T|$ whenever $i_1\neq i_2$.  Also,   $|N(\alpha^{i_1})\cap N(\alpha^{i_2}\beta)|=2|(i_2-i_1+R)\cap T|=2|R\cap T|$. By the arbitrariness of $i_1,i_2\in 2\mathbb{Z}_{2n}$, we assert that every pair of vertices in $H$ have exactly $2|R\cap T|$ common neighbors in $\Gamma$.  Similarly, one can check that every pair of vertices in $\alpha H$ have exactly $2|R\cap T|$ common neighbors in $\Gamma$. Moreover, by setting $i=n$ in \eqref{equ::intersec_set}, we obtain $|R\cap T|=\frac{1}{2}(|R\cap (n+R)|+|T|)>0$.  Hence, every pair of vertices in $H$ (or $\alpha H$) are at distance $2$ in $\Gamma$.
On the other hand,  $\Gamma$  cannot be a complete bipartite graph because $|R\cap T|<n$, and so must be of diameter $3$. Therefore, we conclude that $\Gamma$ is a bipartite distance-regular graph with the intersection array $\{k, k-1,k-\mu; 1,\mu,k\}$, where $k=|R|+|T|$ and $\mu=2|R\cap T|$. Moreover, we claim that $\Gamma$ is non-antipodal and non-trivial, since otherwise $\Gamma$ would be isomorphic to $K_{2n,2n}-2nK_2$, which is impossible by Lemma \ref{lem::key0}.

Conversely, suppose that $\Gamma=\mathrm{Dic}(n,R,T)$ is a  distance-regular dicirculant  not isomorphic to   $K_{4n}$ or  $K_{t\times m}$  ($tm=4n$).  By Lemma \ref{lem::key0},  $\Gamma$ is non-trivial. Furthermore, by Corollary \ref{cor::pri_DRG},  $\Gamma$ is imprimitive. Thus it suffices to consider the following three cases.

{\flushleft \bf Case A.} $\Gamma$ is antipodal but not bipartite.

By Lemma \ref{lem::imprimitive} and Corollary \ref{cor::DRG_dic}, the antipodal quotient $\overline{\Gamma}$ of $\Gamma$ is a primitive distance-regular circulant or dicirculant. Then it follows from Lemma \ref{lem::cir_DRG} and Corollary \ref{cor::pri_DRG} that $\overline{\Gamma}$ is a complete graph, a cycle of prime order, or a Paley graph of prime order.  If $\overline{\Gamma}$ is a cycle of prime order, then $\Gamma$ would be a cycle of order at least $8$, which is impossible because $\mathrm{Dic}_n$ cannot be generated by an inverse-closed subset of size two when $n>1$. If $\overline{\Gamma}$ is a Paley graph of prime order, by  Lemma \ref{lem::confer}, we also deduce a contradiction. Thus  $\overline{\Gamma}$ is a complete graph, and so $d=2$ or $3$ according to Lemma \ref{lem::imprimitive}. By Lemma \ref{lem::key1}, $d\neq 3$, whence $d=2$. However, complete multipartite graphs are the only antipodal distance-regular graphs with diameter $2$. Therefore, there are no non-trivial distance-regular dicirculants which are antipodal but not bipartite.

{\flushleft \bf Case B.} $\Gamma$ is antipodal and bipartite.

By Corollary \ref{cor::DRG_dic}, the antipodal quotient $\overline{\Gamma}$ and the halved graph $\frac{1}{2}\Gamma$ are distance-regular circulants or dicirculants. If $d$ is odd, by Lemma \ref{lem::imprimitive}, $\overline{\Gamma}$ is primitive. As in Case A, we assert that $\overline{\Gamma}$ is a complete graph. Hence, $d=3$. Considering that $\Gamma$ is antipodal and bipartite, we obtain  $\Gamma\cong K_{2n,2n}-2nK_2$, which is impossible because $\Gamma$ is non-trivial. Therefore, we may assume that $d$ is even. Then, by Lemma \ref{lem::imprimitive}, $\frac{1}{2}\Gamma$ is an  antipodal non-bipartite distance-regular circulant or dicirculant with diameter $d_{\frac{1}{2}\Gamma}=\frac{d}{2}$. Clearly, $d\neq 2$. Furthermore, by Lemma \ref{lem::key2}, $d\neq 4$. Therefore, we have $d_{\frac{1}{2}\Gamma}\geq 3$. According to the conclusion of Case A and Lemma \ref{lem::key0}, we assert that $\frac{1}{2}\Gamma$  is a circulant. 
However, by Lemma \ref{lem::cir_DRG}, this is impossible because $\frac{1}{2}\Gamma$ is antipodal, non-bipartite, and has diameter at least $3$. 

{\flushleft \bf Case C.} $\Gamma$ is bipartite but not antipodal.

By Lemma \ref{lem::imprimitive} and Corollary \ref{cor::DRG_dic}, $\frac{1}{2}\Gamma$ is a primitive distance-regular circulant or dicirculant. As in Case A,  $\frac{1}{2}\Gamma$ is a complete graph, a cycle of prime order, or a Paley graph of prime order. We claim that the later two cases cannot occur, since  $\frac{1}{2}\Gamma$ has $2n$ vertices. Thus  $\frac{1}{2}\Gamma\cong K_{2n}$, and $d=2$ or $3$. If $d=2$, then $\Gamma$ is a complete bipartite graph, contrary to our assumption. Hence, $\Gamma$ is a   non-antipodal bipartite non-trivial distance-regular graph with diameter $3$. Recall that $\Gamma=\mathrm{Dic}(n,R,T)=\mathrm{Cay}(\mathrm{Dic}_n,\alpha^R\cup \alpha^T\beta)$ where  $R=-R$ and $T=n+T$. Let $H$ be the bipartition set of $\Gamma$ containing the identity $1\in \mathrm{Dic}_n$. Note that $H=\mathcal{N}_0\cup \mathcal{N}_2$. By Lemma \ref{lem::block}, $H$ is a subgroup of $\mathrm{Dic}_n$ with index $2$.  Observe that  $\mathrm{Dic}_n$ has a unique subgroup of index $2$, namely $\langle\alpha\rangle$, if $n$ is odd, and has two more subgroups of index $2$, namely $\langle\alpha^2,\beta\rangle$ and $\langle\alpha^2,\alpha\beta\rangle$, if $n$ is even. Thus we only need to consider the following three situations.
{\flushleft \bf Subcase C.1.} $H=\langle\alpha\rangle$.

In this situation,  $R=\emptyset$, and $N(1)=\alpha^T\beta$. By Lemma \ref{lem::common_neighbor},  $|N(1)\cap N(\alpha^i)|=|T\cap (i+T)|$ for all $i\in \mathbb{Z}_{2n}$. For every $i\in \mathbb{Z}_{2n}\setminus\{0\}$,  since $\partial(1,\alpha^i)=2$, we have 
$$
\begin{aligned}
|T\cap (i+T)|=|N(1)\cap N(\alpha^i)|=\mu=|N(1)\cap N(\alpha^n)|=|T\cap (n+T)|=|T|=|N(1)|.
\end{aligned}
$$
This implies that $N(1)=N(\alpha^i)$ for all $i\in \mathbb{Z}_{2n}\setminus\{0\}$. Thus, $\Gamma$ is a  complete bipartite graph, contrary to our assumption.

{\flushleft \bf Subcase C.2.} $n$ is even, and $H=\langle\alpha^2,\beta\rangle$.

In this situation, $R$ and $T$ are non-empty subsets of $1+2\mathbb{Z}_{2n}$.  Recall that each pair of vertices in $H$ are at distance $2$ in $\Gamma$. Let $i,j\in 2\mathbb{Z}_{2n}\setminus\{0\}$. Consider the vertices $1$, $\alpha^i$, $\beta$ and $\alpha^j\beta$ of $H$. By Lemma \ref{lem::common_neighbor}, we have
$$
\begin{aligned}
\mu&=|N(1)\cap N(\alpha^i)|=|R\cap (i+R)|+|T\cap (i+T)|\\
&=|N(1)\cap N(\alpha^j\beta)|=2|(j+R)\cap T|\\
&=|N(1)\cap N(\beta)|=2|R\cap T|.
\end{aligned}
$$ 
Hence, we conclude that $|R\cap (i+R)|+|T\cap (i+T)|=2|(j+R)\cap T|=2|R\cap T|$ for all $i,j\in 2\mathbb{Z}_{2n}\setminus\{0\}$. Also note that $|R\cap T|<n$ because $\Gamma$ is not a complete bipartite graph.  The results follows.

{\flushleft \bf Subcase C.3.} $n$ is even, and $H=\langle\alpha^2,\alpha\beta\rangle$.

In this situation, $R$ and $T$ are non-empty subsets of $1+2\mathbb{Z}_{2n}$ and  $2\mathbb{Z}_{2n}$, respectively. Let $T'=1+T$. We see that $T'\subseteq 1+2\mathbb{Z}_{2n}$ and $\Gamma=\mathrm{Dic}(n,R,T)\cong \mathrm{Dic}(n,R,T')=\Gamma'$, where the corresponding graph isomorphism  $f: V(\Gamma)\rightarrow V(\Gamma')$ is defined by $f(\alpha^i)=\alpha^i$ and $f(\alpha^i\beta)=\alpha^{i+1}\beta$. Let $H'$ denote the bipartition set of $\Gamma'$ containing the identity $1\in \mathrm{Dic}_n$. It is easy to see that $H'=\langle\alpha^2,\beta\rangle$.  Using  $\Gamma'$ instead of $\Gamma$, we reduce the situation to Subcase C.2 immediately. 

We complete the proof.
\end{proof}

\begin{remark}\label{remark::main}
\emph{Let $\Gamma=\mathrm{Dic}(n,R,T)$  denote a  bipartite non-trivial distance-regular dicirculant with diameter $3$, and let $H$ denote the  bipartition set of $\Gamma$  containing the identity $1\in \mathrm{Dic}_n$. In \cite[Proposition 5.3]{DJ21}, van Dam and Jazaeri proved that if  $n$ is odd or $H$ is cyclic then $\Gamma$ does not exist, which coincides with the conclusion of Theorem \ref{thm::main} according to the discussion in Case C. Moreover, when $n$ is even and $H$ is not cyclic, they showed that  the subgraph $\mathrm{Cay}(\langle \alpha \rangle, \alpha^R)$ is actually the incidence graph of a partial geometric design with specific  parameters (cf. \cite[Proposition 5.4]{DJ21}).}
\end{remark}

\section{Further research}\label{section::4}

Let $A$ be an abelian group of order $2n$ with exactly one element $\gamma$ of order $2$. The \textit{generalized dicyclic group} $\mathrm{Dic}(A,\beta)$ is defined as the group generated by $A$ and $\beta$ where $\beta^2=\gamma$ and $\beta^{-1}\alpha\beta=\alpha^{-1}$ for all $\alpha\in A$ (see \cite[p. 229]{M94} or \cite[p. 392]{S57}).  In particular, if $A$ is a cyclic group of order $2n$, then the generalized dicyclic group $\mathrm{Dic}(A,\beta)$ coincides with the dicyclic group $\mathrm{Dic}_n$. Naturally, we  propose the following problem.
\begin{problem}\label{prob::main1}
Determine all distance-regular Cayley graphs on generalized dicyclic groups.
\end{problem}

In \cite{HD22}, the authors  determined all distance-regular Cayley graphs on generalized dicyclic groups under the condition that the connection set is minimal with respect to some element, and pointed out that complete graphs are the only  primitive distance-regular Cayley graphs on generalized dicyclic groups.

\section*{Acknowledgements}
X. Huang is supported by National Natural Science Foundation of China (Grant No. 11901540). K. C. Das is supported by National Research Foundation funded by the Korean government (Grant No. 2021R1F1A1050646). L. Lu is supported by National Natural Science Foundation of China (Grant No. 12001544) and  Natural Science Foundation of Hunan Province  (Grant No. 2021JJ40707).

\end{document}